\documentclass[preprint,12pt,times]{elsarticle}

\usepackage{amssymb}
\usepackage{amsmath}
\usepackage{amsthm}
\usepackage{mathtools}
\numberwithin{equation}{section}
\usepackage{geometry}
\usepackage{setspace}
\usepackage{xcolor}
\usepackage{hyperref}
\usepackage[nameinlink,capitalise]{cleveref}
\hypersetup{
  colorlinks=true,
  linkcolor=blue!55!black,
  citecolor=blue!55!black,
  urlcolor=blue!55!black,
  pdftitle={Metrically bounded Volterra-type operators on area Nevanlinna spaces},
  pdfauthor={Cezhong Tong, Maofa Wang, Zixing Yuan}
}

\newtheorem{theorem}{Theorem}[section]

\newtheorem{proposition}[theorem]{Proposition}
\newtheorem{lemma}[theorem]{Lemma}

\crefname{theorem}{Theorem}{Theorems}
\Crefname{theorem}{Theorem}{Theorems}
\crefname{corollary}{Corollary}{Corollaries}
\Crefname{corollary}{Corollary}{Corollaries}
\crefname{proposition}{Proposition}{Propositions}
\Crefname{proposition}{Proposition}{Propositions}
\crefname{lemma}{Lemma}{Lemmas}
\Crefname{lemma}{Lemma}{Lemmas}
\crefalias{corollary}{corollary}
\crefalias{proposition}{proposition}
\crefalias{lemma}{lemma}

\newcommand{\D}{\mathbb D}
\newcommand{\norm}[1]{\left\lVert #1\right\rVert}
\newcommand{\mb}{\mathrm{mb}}

\begin{document}

\begin{frontmatter}

\title{Metrically Bounded Volterra-Type Operators\\ on Area Nevanlinna Spaces}

\author[aff1]{Zixing Yuan}
\ead{yuan980127@163.com}

\affiliation[aff1]{organization={Institute of Mathematics, Hebei University of Technology},
city={Tianjin}, postcode={300401}, country={China}}

\begin{abstract}
In this paper, we study metrically bounded Volterra-type operators $J_g$ and $I_g$
on the area Nevanlinna spaces $N_\alpha^p$, where $1\le p<\infty$ and $\alpha>-1$.
Choe, Koo and Smith \cite{CKS} obtained partial characterizations of the corresponding symbol classes, but the exact characterization was still open. We establish a Littlewood--Paley type characterization of $N_\alpha^p$ and use it to resolve this problem affirmatively. More precisely, $J_g$ is metrically bounded on
$N_\alpha^p$ if and only if $g$ belongs to the Bloch space $\mathcal B$, while
$I_g$ is metrically bounded on $N_\alpha^p$ if and only if $g\in H^\infty$.
\end{abstract}

\begin{keyword}
Area Nevanlinna space \sep metrically bounded operator \sep Volterra-type operator

\MSC[2020] 32A35 \sep 47G10 \sep 47B38
\end{keyword}

\end{frontmatter}

\thispagestyle{plain}

\section{Introduction}\label{sec:introduction}
The classical Nevanlinna class is defined in terms of boundary means of $\log^+|f|$, while area Nevanlinna spaces are defined by replacing these boundary means with weighted area integrals over the unit disk.
 The operator theoretic properties of area Nevanlinna spaces, including  composition operators, weighted composition operators, and integral operators, as well as the algebraic structure of these spaces, have been extensively studied.
 Xiao \cite{Xiao} studied compact composition
operators on the area Nevanlinna class. Haldimann and Jarchow \cite{HJ}
investigated weighted Nevanlinna algebras, while Jarchow and Xiao \cite{JX}
studied  composition operators between weighted Nevanlinna and Bergman
spaces. Choe, Koo and Smith \cite{CKS} characterized Carleson measures for area Nevanlinna spaces, with applications to composition operators and Volterra-type operators. For further related results, see  \cite{WangVector,WangLiu,YZ} and the references therein.

Let $\D=\{z\in\mathbb C:|z|<1\}$, and let $H(\D)$ denote the space of
analytic functions on $\D$. For $1\le p<\infty$ and $\alpha>-1$, 
 the area Nevanlinna space $N_\alpha^p$ consists of all
$f\in H(\D)$ such that
\begin{equation}\label{eq:Npa-functional}
	\norm{f}_{N_\alpha^p}
	=\left\{\int_\D[\log(1+|f(z)|)]^p\,dA_\alpha(z)\right\}^{1/p}<\infty,
\end{equation}
where $dA_\alpha(z)=(1+\alpha)(1-|z|^2)^\alpha dA(z)$ and
$dA(z)=dx\,dy/\pi$ is the normalized area measure.

Notice that \eqref{eq:Npa-functional} is not a norm in the usual sense.
We nevertheless use the notation
$\norm{\cdot}_{N_\alpha^p}$ for convenience. Since
$\log(1+x+y)\le \log(1+x)+\log(1+y)$ for $x,y\ge0$, Minkowski's
inequality shows that \eqref{eq:Npa-functional} induces the
translation-invariant metric
\[
d(f,g)=\norm{f-g}_{N_\alpha^p},\qquad f,g\in N_\alpha^p.
\]
A linear operator $T$ on $N_\alpha^p$ is called \emph{metrically bounded} if
there exists $C>0$ such that
\[
 \norm{Tf}_{N_\alpha^p}\le C\norm{f}_{N_\alpha^p},
 \qquad f\in N_\alpha^p.
\]

For $g\in H(\D)$, the Volterra-type operators $J_g$ and $I_g$ are defined by
\begin{equation}\label{eq:operators}
 J_gf(z)=\int_0^z f(\zeta)g'(\zeta)\,\mathrm{d}\zeta,
 \qquad
 I_gf(z)=\int_0^z f'(\zeta)g(\zeta)\,\mathrm{d}\zeta,
\end{equation}
for $z\in\D$ and $f\in H(\D)$. The discussion of Volterra-type operators
$J_g$ and $I_g$ first arose in connection with semigroups of composition
operators, for further details and background, see
\cite{ASHp,AS,Miihkinen,Pommerenke,SiskakisSurvey,SiskakisZhao}. In this
paper, we focus on the metric boundedness of \eqref{eq:operators} on
$N_\alpha^p$.

Define
\[
\begin{aligned}
 J_{\mb}[N_\alpha^p]
 &=\{g\in H(\D):J_g\text{ is metrically bounded on }N_\alpha^p\},\\
 I_{\mb}[N_\alpha^p]
 &=\{g\in H(\D):I_g\text{ is metrically bounded on }N_\alpha^p\}.
\end{aligned}
\]
We also denote by $H^\infty$ the space of bounded analytic functions on $\mathbb D$, and by
$\mathcal B=\{g\in H(\D):\sup_{z\in\D}(1-|z|^2)|g'(z)|<\infty\}$
the Bloch space.

Choe, Koo and Smith \cite[Proposition~4.3]{CKS} proved that
\begin{equation}\label{eq:known-symbol-inclusions}
 H_1^\infty\subset J_{\mb}[N_\alpha^p]\subset\mathcal B,
 \qquad
 H_1^\infty\subset I_{\mb}[N_\alpha^p]\subset H^\infty,
\end{equation}
where $H_1^\infty=\{g\in H^\infty:g'\in H^\infty\}$.
In attempting to identify the symbol classes in
\eqref{eq:known-symbol-inclusions}, they were led in
\cite[Remarks~4.4(3)]{CKS} to ask whether the additive constant in their
estimate could be dropped. Equivalently, whether
\begin{equation}\label{eq:open-homogeneous}
 \norm{f-f(0)}_{N_\alpha^p}^{p}
 \le C_{p,\alpha}
 \int_\D
 \bigl[\log(1+(1-|z|)|f'(z)|)\bigr]^p\,dA_\alpha(z),
 \qquad f\in H(\D)
\end{equation}
holds. They had already proved the reverse estimate
\begin{equation}\label{eq:CKS-direct}
 \int_\D
 \bigl[\log(1+(1-|z|)|f'(z)|)\bigr]^p\,dA_\alpha(z)
 \le C_{p,\alpha}\norm{f-f(0)}_{N_\alpha^p}^p
\end{equation}
in \cite[Remarks~4.4(3)]{CKS}, as a consequence of their estimate~(3.1).
Thus, together with \eqref{eq:CKS-direct}, an affirmative answer to
\eqref{eq:open-homogeneous} would yield a Littlewood-Paley type
characterization of $N_\alpha^p$. Our first main result gives such an
affirmative answer.

\begin{theorem}\label{thm:main}
For $1\le p<\infty$ and $\alpha>-1$. There exists a constant
$C=C(p,\alpha)\ge1$ such that, for every $f\in H(\D)$,
\begin{equation}\label{eq:main-two-sided}
 C^{-1}\norm{f-f(0)}_{N_\alpha^p}^{p}
 \le
 \int_\D
 \bigl[\log(1+(1-|z|)|f'(z)|)\bigr]^{p}\,dA_\alpha(z)
 \le
 C\norm{f-f(0)}_{N_\alpha^p}^{p}.
\end{equation}
\end{theorem}

Theorem~\ref{thm:main} also shows that the lower inclusions in
\eqref{eq:known-symbol-inclusions} are strict. Indeed, let
$g_1(z)=\log\frac{1}{1-z}$. Since
$
 (1-|z|)|g_1'(z)|\le1
$
and $(J_{g_1}f)'=fg_1'$, Theorem~\ref{thm:main} gives
\[
 \norm{J_{g_1}f}_{N_\alpha^p}^{p}
 \lesssim
 \int_\D[\log(1+|f(z)|)]^p\,dA_\alpha(z)
 =\norm{f}_{N_\alpha^p}^{p}.
\]
Hence $g_1\in J_{\mb}[N_\alpha^p]$, while $g_1\notin H_1^\infty$.
Similarly, for $0<\gamma<1$, let $g_2(z)=(1-z)^\gamma$. Since
$g_2\in H^\infty$ and $(I_{g_2}f)'=f'g_2$, Theorem~\ref{thm:main} yields
\[
 \norm{I_{g_2}f}_{N_\alpha^p}^{p}
 \lesssim
 \int_\D
 \bigl[\log(1+(1-|z|)|f'(z)|)\bigr]^p\,dA_\alpha(z)
 \lesssim \norm{f}_{N_\alpha^p}^{p}.
\]
On the other hand,
$|g_2'(r)|=\gamma(1-r)^{\gamma-1}\to\infty$ as $r\to1^-$, so
$g_2\notin H_1^\infty$. Consequently,
\[
 H_1^\infty\subsetneq J_{\mb}[N_\alpha^p],
 \qquad
 H_1^\infty\subsetneq I_{\mb}[N_\alpha^p].
\]
It is therefore natural to ask whether the upper inclusions in
\eqref{eq:known-symbol-inclusions} are equalities. Our second main result
gives an affirmative answer.

\begin{theorem}\label{thm:symbol-classes}
For $1\le p<\infty$ and $\alpha>-1$,
\[
 J_{\mb}[N_\alpha^p]=\mathcal B,
 \qquad
 I_{\mb}[N_\alpha^p]=H^\infty.
\]
\end{theorem}

The paper is organized as follows. Section~\ref{sec:auxiliary} collects the
preliminary lemmas. Section~\ref{sec:derivative} proves  Theorems~\ref{thm:main} and \ref{thm:symbol-classes}.

Throughout the paper, we write \(A\lesssim B\) if \(A\le CB\) for some constant \(C>0\), and
\(A\asymp B\) if both \(A\lesssim B\) and \(B\lesssim A\) hold.

\section{Preliminaries}\label{sec:auxiliary}
The following lemmas needed in the proofs.
\begin{lemma}\label{lem:weighted-hardy}
Let $1\le p<\infty$ and $0<q<1$. For every nonnegative sequence $(x_k)_{k\ge0}$,
\[
 \sum_{N=0}^{\infty}q^{N}
 \left(\sum_{k=0}^{N}x_k\right)^{p}
 \le
 \frac{1}{(1-q^{1/p})^{p}}
 \sum_{k=0}^{\infty}q^{k}x_k^{p}.
\]
\end{lemma}

\begin{proof}
Define sequences on $\mathbb Z$ by
\[
 y_k=\begin{cases}q^{k/p}x_k,&k\ge0,\\0,&k<0,\end{cases}
 \qquad
 h_j=\begin{cases}q^{j/p},&j\ge0,\\0,&j<0.\end{cases}
\]
For every $N\ge0$, we have
\[
\begin{aligned}
 (h*y)_N=\sum_{k\in\mathbb Z}h_{N-k}y_k=q^{N/p}\sum_{k=0}^{N}x_k.
\end{aligned}
\]
Young's convolution inequality on $\ell^{p}(\mathbb Z)$ gives
$
 \norm{h*y}_{\ell^{p}}
 \le \norm{h}_{\ell^{1}}\norm{y}_{\ell^{p}}.
$
Therefore,
\[
 \left\{
 \sum_{N=0}^{\infty}q^{N}
 \left(\sum_{k=0}^{N}x_k\right)^{p}
 \right\}^{1/p}
 \le
 \frac{1}{1-q^{1/p}}
 \left(\sum_{k=0}^{\infty}q^{k}x_k^{p}\right)^{1/p},
\] which completes the proof.
\end{proof}

\begin{lemma}\label{lem:subharmonic-log}
For $1\le p<\infty$. Let $h\in H(\D)$ and $\delta>0$. Then
\[
 z\longmapsto \bigl[\log(1+\delta|h(z)|)\bigr]^{p}
\]
is a nonnegative subharmonic function on $\D$.
\end{lemma}

\begin{proof}
If $h\equiv0$, the result is immediate. Suppose that $h\not\equiv0$.
By \cite[Theorem~17.3]{Rudin}, the function
\[
 u(z)=\log|h(z)|,\qquad \log0:=-\infty,
\]
is subharmonic in $\D$. Set
\[
 \psi(t)=\bigl[\log(1+\delta e^t)\bigr]^p,
 \qquad t\in\mathbb R,
\]
and $\psi(-\infty)=0$. Then
\[
 \psi'(t)
 =
 p\bigl[\log(1+\delta e^t)\bigr]^{p-1}
 \frac{\delta e^t}{1+\delta e^t}>0,
\]
and
\[
 \psi''(t)
 =
 \frac{
 p\delta e^t
 \bigl[\log(1+\delta e^t)\bigr]^{p-2}
 }{(1+\delta e^t)^2}
 \left[
 (p-1)\delta e^t+\log(1+\delta e^t)
 \right]
 \ge0.
\]
Hence $\psi$ is increasing and convex. By
\cite[Theorem~17.2]{Rudin}, we know that 
$
 \psi(u(z))
 =
 \bigl[\log(1+\delta|h(z)|)\bigr]^p
$
is subharmonic in $\D$.
\end{proof}

\section{Proof  of Theorems~\ref{thm:main} and \ref{thm:symbol-classes}.}\label{sec:derivative}

For $f\in H(\D)$, we first prove an upper estimate for $\norm{f-f(0)}_{N_\alpha^p}^p$, which will be a key tool in the proof of our main results. Set
\[
 I_k=[1-2^{-k},\,1-2^{-k-1}),\qquad k\ge0.
\]
For $\theta\in[0,2\pi)$ and $k\ge0$, define
\[
 b_k(\theta)
 =\sup_{t\in I_k}
 \log(1+2^{-k}|f'(te^{i\theta})|).
\]

\begin{proposition}\label{prop:radial-reduction}
For $1\le p<\infty$, $\alpha>-1$, and $f\in H(\D)$. Then
\begin{equation}\label{eq:radial-reduction}
 \norm{f-f(0)}_{N_{\alpha}^{p}}^{p}
 \le C_{p,\alpha}
 \sum_{k=0}^{\infty}2^{-k(\alpha+1)}
 \int_{0}^{2\pi}b_k(\theta)^{p}\,d\theta.
\end{equation}
\end{proposition}

\begin{proof}
Fix $\theta$, let $N\ge0$ be an integer, and take $r\in I_N$. Then it is easy to see that
\[
 |f(re^{i\theta})-f(0)|
 \le \sum_{k=0}^{N}\int_{I_k}|f'(te^{i\theta})|\,dt.
\]
Since $|I_k|=2^{-k-1}$, for every $k\ge0$, we get
\[
 1+\int_{I_k}|f'(te^{i\theta})|\,dt
 \le 1+2^{-k}\sup_{t\in I_k}|f'(te^{i\theta})|
 \le e^{b_k(\theta)}.
\]
Using the elementary inequality
\[
 1+\sum_{k=0}^{N}s_k\le\prod_{k=0}^{N}(1+s_k),
 \qquad s_k\ge0,
\]
we obtain
\begin{equation}\label{eq:pointwise-layer-bound}
 \log(1+|f(re^{i\theta})-f(0)|)
 \le \sum_{k=0}^{N}b_k(\theta),
 \qquad r\in I_N.
\end{equation}

For $N\ge1$ and $r\in I_N$, we have
\begin{equation}\label{eq:layer-mass}
 \int_{I_N}r(1-r^2)^\alpha\,dr
 \asymp_\alpha 2^{-N(\alpha+1)}.
\end{equation}
Thus, for $M\ge0$, using polar coordinates, \eqref{eq:pointwise-layer-bound}, and \eqref{eq:layer-mass} gives,
\begin{align*}
 &\int_{|z|<1-2^{-M-1}}
 [\log(1+|f(z)-f(0)|)]^p\,dA_\alpha(z)\le C_\alpha\sum_{N=0}^{M}2^{-N(\alpha+1)}
 \int_0^{2\pi}\left(\sum_{k=0}^{N}b_k(\theta)\right)^p d\theta.
\end{align*}
For each fixed $\theta$, taking 
$q=2^{-(\alpha+1)}$ in Lemma~\ref{lem:weighted-hardy}, we get
\[
 \sum_{N=0}^{M}2^{-N(\alpha+1)}\left(\sum_{k=0}^{N}b_k(\theta)\right)^p
 \le C_{p,\alpha}\sum_{k=0}^{M}2^{-k(\alpha+1)} b_k(\theta)^p.
\]
Therefore,
\[
 \int_{|z|<1-2^{-M-1}}
 [\log(1+|f(z)-f(0)|)]^p\,dA_\alpha(z)
 \le C_{p,\alpha}\sum_{k=0}^{M}2^{-k(\alpha+1)}
 \int_0^{2\pi}b_k(\theta)^p\,d\theta.
\]
Letting $M\to\infty$ and using monotone convergence theorem proves \eqref{eq:radial-reduction}.
\end{proof}

We now prove Theorem \ref{thm:main}.

\begin{proof}[Proof of \cref{thm:main}]
By Proposition~\ref{prop:radial-reduction}, it remains to prove the following inequality
\begin{equation}\label{eq:maxima-to-area}
 \sum_{k=0}^{\infty}2^{-k(\alpha+1)}
 \int_{0}^{2\pi}b_k(\theta)^p\,d\theta
 \le C_{p,\alpha}
 \int_{\D}
 \bigl[\log(1+(1-|z|)|f'(z)|)\bigr]^p
 \,dA_\alpha(z).
\end{equation}

Let $k\ge0$ and fix a sufficiently small absolute constant
$c\in(0,1/2)$. For $\theta\in[0,2\pi)$, let
\[
T_{k,\theta}
=\left\{w\in\D:
\operatorname{dist}\!\left(
w,\{te^{i\theta}:t\in I_k\}
\right)<c\,2^{-k}\right\}.
\]
For every $z=te^{i\theta}$ with $t\in I_k$, we have
$
1-|z|>2^{-k-1}>c\,2^{-k},
$
and so
$
D(z,c\,2^{-k})\subset\D.
$
Moreover, by the definition of $T_{k,\theta}$, we also have
$
D(z,c\,2^{-k})\subset T_{k,\theta}.
$
Therefore, by Lemma~\ref{lem:subharmonic-log} and the subharmonic
mean inequality, we obtain
\begin{equation}\label{eq:bk-tube}
\begin{aligned}
b_k(\theta)^p
&=\sup_{t\in I_k}
\bigl[\log(1+2^{-k}|f'(te^{i\theta})|)\bigr]^p\\
&\le
\sup_{t\in I_k}
\frac{C}{(c\,2^{-k})^2}
\int_{D(te^{i\theta},c\,2^{-k})}
\bigl[\log(1+2^{-k}|f'(w)|)\bigr]^p\,dA(w)\\
&\le C\,2^{2k}\int_{T_{k,\theta}}
\bigl[\log(1+2^{-k}|f'(w)|)\bigr]^p\,dA(w).
\end{aligned}
\end{equation}

If $w\in T_{k,\theta}$, then there exists $t\in I_k$ such that
$
\bigl||w|-t\bigr|
\le |w-te^{i\theta}|<c\,2^{-k}.
$
Hence
\[
\left(\frac12-c\right)2^{-k}
<1-|w|<(1+c)2^{-k},
\]
so $w\in\widehat A_k$, where
$
\widehat A_k
=\left\{w\in\D:
\left(\frac12-c\right)2^{-k}<1-|w|<(1+c)2^{-k}
\right\}.
$

For $k\ge1$, if $w\in T_{k,\theta}$, choose $t\in I_k$ such that
$|w-te^{i\theta}|<c\,2^{-k}$. Then
\[
|w|\,|e^{i\theta}-e^{i\arg w}|
\le \bigl||w|-t\bigr|+|te^{i\theta}-w|
<2c\,2^{-k},
\]
and
\[
|w|
\ge t-c\,2^{-k}
\ge \frac12-\frac c2
=\frac{1-c}{2}>0.
\]
Hence
$
\min_{m\in\mathbb Z}
|\theta-\arg w-2\pi m|
\le \frac{\pi}{2}|e^{i\theta}-e^{i\arg w}|
\le C\,2^{-k}.
$
For $k=0$,
\[
\bigl|\{\theta\in[0,2\pi):w\in T_{0,\theta}\}\bigr|
\le 2\pi.
\]
Thus, for all $k\ge0$, we have
$
\bigl|\{\theta\in[0,2\pi):w\in T_{k,\theta}\}\bigr|
\le C\,2^{-k}.
$

Integrating \eqref{eq:bk-tube} in $\theta$ and applying Tonelli's theorem, we obtain
\begin{align*}
2^{-k(\alpha+1)}\int_0^{2\pi}b_k(\theta)^p\,d\theta
&\le C\,2^{k-k\alpha}
\int_{\widehat A_k}\bigl[\log(1+2^{-k}|f'(w)|)\bigr]^p
\bigl|\{\theta\in[0,2\pi):w\in T_{k,\theta}\}\bigr|\,dA(w)\\
&\le C\,2^{-k\alpha}\int_{\widehat A_k}
\bigl[\log(1+2^{-k}|f'(w)|)\bigr]^p\,dA(w).
\end{align*}
On $\widehat A_k$, it is easy to see that
$
2^{-k\alpha}\asymp_\alpha(1-|w|^2)^\alpha.
$
Also, for $s\ge0$ and $\lambda>0$,
\begin{align}\label{A1}
\log(1+\lambda s)\le\max\{1,\lambda\}\log(1+s).
\end{align}
For $0<\lambda\le1$ this follows from monotonicity, while for
$\lambda\ge1$ it follows from $1+\lambda s\le(1+s)^\lambda$.
Together with $2^{-k}\asymp1-|w|$, this implies
\begin{equation}\label{eq:single-layer-final}
2^{-k(\alpha+1)}\int_0^{2\pi}b_k(\theta)^p\,d\theta
\le C_{p,\alpha}\int_{\widehat A_k}
[\log(1+(1-|w|)|f'(w)|)]^p\,dA_\alpha(w).
\end{equation}
For $w\in\widehat A_k$,
\[
\frac{1-|w|}{1+c}<2^{-k}
<\frac{1-|w|}{1/2-c},
\qquad
\sum_{k=0}^{\infty}\chi_{\widehat A_k}(w)\le C,
\quad w\in\D.
\]
Hence, by \eqref{eq:single-layer-final},
\[
\begin{aligned}
\sum_{k=0}^{\infty}
2^{-k(\alpha+1)}
\int_0^{2\pi}b_k(\theta)^p\,d\theta
&\le
C_{p,\alpha}
\sum_{k=0}^{\infty}
\int_{\widehat A_k}
[\log(1+(1-|w|)|f'(w)|)]^p\,dA_\alpha(w)\\
&=
C_{p,\alpha}
\int_{\D}
[\log(1+(1-|w|)|f'(w)|)]^p
\sum_{k=0}^{\infty}\chi_{\widehat A_k}(w)\,dA_\alpha(w)\\
&\le
C_{p,\alpha}
\int_{\D}
[\log(1+(1-|w|)|f'(w)|)]^p\,dA_\alpha(w).
\end{aligned}
\]
This proves \eqref{eq:maxima-to-area}.

Combining \eqref{eq:radial-reduction} and \eqref{eq:maxima-to-area}, we obtain
\[
\norm{f-f(0)}_{N_{\alpha}^{p}}^p
\le C_{p,\alpha}
\int_{\D}
\bigl[\log(1+(1-|z|)|f'(z)|)\bigr]^p
\,dA_\alpha(z).
\]
This is the left-hand inequality in \eqref{eq:main-two-sided}. The right-hand inequality is \eqref{eq:CKS-direct}, proved in \cite[Remarks~4.4(3)]{CKS}.
\end{proof}

We now prove Theorem~\ref{thm:symbol-classes}. 

\begin{proof}[Proof of Theorem~\ref{thm:symbol-classes}]
%We shall use the elementary inequality
%\begin{align}\label{A1}
% \log(1+\lambda s)\le\max\{1,\lambda\}\log(1+s),
% \qquad s\ge0,\ \lambda>0.
%\end{align}
%For $0<\lambda\le1$ this follows from monotonicity, and for $\lambda\ge1$ it follows from $1+\lambda s\le(1+s)^\lambda$.

\textbf{The class $J_{\mb}[N_\alpha^p]$:}
Let $g\in\mathcal B$, set
$\beta_g=\sup_{z\in\D}(1-|z|^2)|g'(z)|$, and fix $f\in N_{\alpha}^{p}$. 
The left-hand inequality in \eqref{eq:main-two-sided}, applied to $J_gf$, gives
\begin{align*}
 \norm{J_gf}_{N_{\alpha}^{p}}^p
 &\le C_{p,\alpha}
 \int_{\D}
 \bigl[\log(1+(1-|z|)|f(z)g'(z)|)\bigr]^p
 \,dA_\alpha(z)\\
 &\le C_{p,\alpha}
 \int_{\D}
 \bigl[\log(1+\beta_g|f(z)|)\bigr]^p
 \,dA_\alpha(z).
\end{align*}
By \eqref{A1}, we have
\[
 \log(1+\beta_g|f(z)|)
 \le\max\{1,\beta_g\}\log(1+|f(z)|).
\]
Therefore
\[
 \norm{J_gf}_{N_{\alpha}^{p}}
 \le C_{p,\alpha}^{1/p}\max\{1,\beta_g\}
 \norm{f}_{N_{\alpha}^{p}}.
\]
Thus $g\in J_{\mb}[N_{\alpha}^{p}]$, that is
$
 \mathcal B\subset J_{\mb}[N_{\alpha}^{p}].
$
The reverse inclusion
$
 J_{\mb}[N_{\alpha}^{p}]\subset\mathcal B
$
is \cite[Proposition 4.3(b)]{CKS}. Hence
\[
 J_{\mb}[N_{\alpha}^{p}]=\mathcal B.
\]

\textbf{The class $I_{\mb}[N_\alpha^p]$:}
Let $g\in H^{\infty}$ and fix $f\in N_\alpha^p$.  Similar to the proof of $J_g$, by 
 \eqref{eq:main-two-sided}, we get
\begin{align*}
 \norm{I_gf}_{N_{\alpha}^{p}}^p
 &\le C_{p,\alpha}
 \int_{\D}
 \bigl[\log(1+(1-|z|)|g(z)f'(z)|)\bigr]^p
 \,dA_\alpha(z)\\
 &\le C_{p,\alpha}\max\{1,\norm{g}_{\infty}\}^p
 \int_{\D}
 \bigl[\log(1+(1-|z|)|f'(z)|)\bigr]^p
 \,dA_\alpha(z)\\
 &\le C_{p,\alpha}\max\{1,\norm{g}_{\infty}\}^p
 \norm{f-f(0)}_{N_{\alpha}^{p}}^p.
\end{align*}
Since $[\log(1+|f|)]^p$ is subharmonic, we have
$
 \log(1+|f(0)|)\le \norm{f}_{N_\alpha^p}.
$
Consequently, by the triangle inequality, we obtain
\[
 \norm{f-f(0)}_{N_\alpha^p}
 \le \norm{f}_{N_\alpha^p}+\log(1+|f(0)|)
 \le 2\norm{f}_{N_\alpha^p}.
\]
It follows that
\[
 \norm{I_gf}_{N_{\alpha}^{p}}
 \le C_{p,\alpha}^{1/p}2\max\{1,\norm{g}_{\infty}\}
 \norm{f}_{N_{\alpha}^{p}}.
\]
Thus,
$
 H^{\infty}\subset I_{\mb}[N_{\alpha}^{p}].
$
The reverse inclusion
$
 I_{\mb}[N_{\alpha}^{p}]\subset H^{\infty}
$
is \cite[Proposition 4.3(c)]{CKS}. Therefore
\[
 I_{\mb}[N_{\alpha}^{p}]=H^{\infty}.
\]
This completes the proof.
\end{proof}

\section*{Declarations}

\noindent\textbf{Data availability}
No data were used for the research described in this article.

%\medskip
%\noindent\textbf{Conflict of interest.}
%The authors declare no conflict of interest.

\end{document}